\documentclass{composition}
\usepackage{latexsym,amsmath,amssymb,amsthm,amscd,graphicx,amsfonts}
\usepackage{mathrsfs}
\usepackage{amsmath}
\usepackage{amssymb}
\usepackage{verbatim}
\usepackage{xcolor}
\usepackage{graphicx}
\usepackage[all]{xy}
\numberwithin{equation}{section}

\theoremstyle{plain}
\newtheorem{theorem}{Theorem}[section]
\newtheorem{proposition}{Proposition}[section]
\newtheorem{corollary}{Corollary}[section]

\theoremstyle{definition}
\newtheorem{definition}{Definition}[section]

\theoremstyle{remark}
\newtheorem{rem}{Remark}[section]

\begin{document}

\title{On the global generation of higher direct images of pluricanonical bundles}

\author{Jixiang Fu}
\email{majxfu@fudan.edu.cn}
\address{Shanghai Center for Mathematical Sciences, Fudan University, Shanghai 200433, People's Republic of China}

\author{Jingcao Wu}
\email{wujincao@shufe.deu.cn}
\address{School of Mathematics, Shanghai University of Finance and Economy, Shanghai 200433, People's Republic of China}

\thanks{This research was supported by China Postdoctoral Science Foundation, grant 2019M661328.}

\begin{abstract}
Given a fibration $f$ between two projective manifolds $X$ and $Y$, we discuss the effective generation of the higher direct images $R^{i}f_{\ast}(K^{m}_{X})$, where $K^{m}_{X}$ is the $m$-th tensor power of the canonical bundle of $X$. In particular, we answer two questions posed by Popa--Schnell in \cite{PS14}.
\end{abstract}

\maketitle

\section{Introduction}
\label{sec:introduction}

Assume that $f:X\rightarrow Y$ is a fibration, i.e. a surjective morphism with connected fibres between two projective manifolds $X$ and $Y$. Denote by $K^{m}_{X}$  the $m$-th tensor power of the canonical bundle $K_X$ of $X$. The positivity of the associated higher direct image $R^{i}f_{\ast}(K^{m}_{X})$ is of significant importance for understanding the geometry of this fibration. Fruitful results have been obtained on this subject, such as \cite{Ber08,Ber09,Hor10,Kaw81,Kaw82,Ko86a,Ko86b,Ko87,Vie82b,Vie83}.

Popa and Schnell \cite{PS14} proved the following result inspired by the brilliant work of Viehweg \cite{Vie82b,Vie83} and Koll\'{a}r \cite{Ko86a,Ko86b}:

\begin{theorem}[(Popa--Schnell)]\label{t11}
Let $f:X\rightarrow Y$ be a fibration between two projective manifolds with $\dim Y=n$,
 and $A$ be an ample and globally generated line bundle on $Y$. If $m\geqslant 1$ is an integer, then the sheaf
\begin{equation*}
f_{\ast}(K^{m}_{X})\otimes A^{l}
\end{equation*}
is $0$-regular, and therefore globally generated, for $l\geqslant m(n+1)$.
\end{theorem}

Popa and Schnell then posed a question whether the similar result holds for higher direct images. See {\it Question} in \cite{PS14} after Corollary 2.10. In this paper, we give a positive answer to this question in some sense. Indeed, since by Proposition \ref{p21}
\begin{equation*}
f_{\ast}(K^{m}_{X}\otimes\mathscr{I}(f,\|K^{m-1}_{X}\|))=f_{\ast}(K^{m}_{X}),
\end{equation*}
it is reasonable to involve the asymptotic multiplier ideal when we consider higher direct images. So our main result is as follows, which implies Theorem \ref{t11} when $i=0$.

\begin{theorem}\label{t12}
Let $f:X\rightarrow Y$ be a fibration between two projective manifolds with $\dim Y=n$, and $A$ be an ample and globally generated line bundle on $Y$. If $m\geqslant1$ is an integer, then the sheaf
\begin{equation*}
R^{i}f_{\ast}(K^{m}_{X}\otimes\mathscr{I}(f,\|K^{m-1}_{X}\|))\otimes A^{l}
\end{equation*}
is $0$-regular, and therefore globally generated, for $i\geqslant0$ and $l\geqslant m(n+1)$.
\end{theorem}

We use the strategy in \cite{Ko86a,Ko86b} to prove Theorem \ref{t12}. The idea is expanded as follows. First we prove a Koll\'{a}r-type vanishing theorem.

\begin{theorem}\label{t13}
Let $f:X\rightarrow Y$ be a fibration between two projective manifolds with $\dim Y=n$, and $A$ be an ample and globally generated line bundle on $Y$. If $m\geqslant1$ is an integer, then for any $l\geqslant(m-1)n+m$, $i\geqslant0$ and $q>0$,
\begin{equation*}
H^{q}(Y,R^{i}f_{\ast}(K^{m}_{X}\otimes\mathscr{I}(f,\|K^{m-1}_{X}\|))\otimes A^{l})=0.
\end{equation*}
\end{theorem}

The Koll\'{a}r-type vanishing theorem has been fully studied. See, for example,  \cite{Eno93,Fuj12,FuM16,GoM17,Ko86a,Ko86b,Mat14,Mat16,Ohs84,Wu20}. However, we cannot directly apply any result among these papers to obtain Theorem \ref{t13}.
The reason is that in general there does not exist a metric $\varphi$ on $K_{X}$ such that
\begin{equation*}
i\Theta_{K_{X},\varphi}\geqslant0\quad\textup{and}\quad
\mathscr{I}(\varphi)=\mathscr{I}(f,\|K_{X}\|).
\end{equation*}
We will construct a suitable metric on $K_{X}\otimes f^{\ast}A^{n+1}$ to overcome this problem. The more details are explained in Sect. \ref{sec:vanishing}.

Now Theorem \ref{t12} is a direct consequence of Theorem \ref{t13} combined with the Castelnuovo--Mumford regularity \cite{Mum66}.

After that, we prove a generic vanishing theorem for the higher direct images, which answers another question of Popa and Schnell \cite{PS14} in some sense. See {\sl Question} in \cite{PS14} after Corollary 5.4. Note that T.~Shibata \cite{Shi16} provided an example to show that $R^{i}f_{\ast}(K^{m}_{X})$, $m\geqslant2$, are not necessarily $GV$-sheaves. Hence it is quite natural to consider instead $R^{i}f_{\ast}(K^{m}_{X}\otimes\mathscr{I}(f,\|K^{m-1}_{X}\|))$.

\begin{theorem}\label{t14}
Let $f:X\rightarrow Y$ be a morphism from a projective manifold $X$ to an abelian variety $Y$. Then the sheaf
\begin{equation*}
R^{i}f_{\ast}(K^{m}_{X}\otimes\mathscr{I}(f,\|K^{m-1}_{X}\|))
\end{equation*}
is a $GV$-sheaf \cite{PP11} for every $i\geqslant0$ and $m\geqslant1$.
\end{theorem}

This result leads in turn to the following vanishing and generation results which are stronger than those for morphisms to arbitrary varieties.

\begin{corollary}\label{c11}
If $f:X\rightarrow Y$ is a morphism from a projective manifold to an abelian variety and $A$ is an ample line bundle on $Y$, then for every $i\geqslant0$ and $m\geqslant1$ one has:
\begin{enumerate}
\item[(1)] $R^{i}f_{\ast}(K^{m}_{X}\otimes\mathscr{I}(f,\|K^{m-1}_{X}\|))$ is a nef sheaf on $Y$ (see Sect. \ref{sec:coherent}).
\item[(2)] $H^{q}(Y,R^{i}f_{\ast}(K^{m}_{X}\otimes\mathscr{I}(f,\|K^{m-1}_{X}\|))\otimes A)=0$ for all $q>0$.
\item[(3)] $R^{i}f_{\ast}(K^{m}_{X}\otimes\mathscr{I}(f,\|K^{m-1}_{X}\|))\otimes A^{2}$ is globally generated.
\end{enumerate}
\end{corollary}

In the end, we make some further discussions. Note that after \cite{PS14}, there are several references such as \cite{Den21,DuM19,Dut20,Iwa20} which aims to improve Theorem \ref{t11}. It is remarkable that there is no more global generation required for $A$ in order to give the effective lower bound of $l$ such that $f_{\ast}(K^{m}_{X})\otimes A^{l}$ is (generically) globally generated. Hence it is natural to try to remove the global generation condition in our theorems.

However, since Theorem \ref{t12} depends highly on the Castelnuovo--Mumford regularity, it is not easy to make such an extension. Currently, we can only make the following generalisation of Theorem \ref{t13} and also prove a Koll\'{a}r-type injectivity theorem. These results seem to be of independent interest.

\begin{theorem}\label{t15}
Let $f:X\rightarrow Y$ be a smooth fibration between two projective manifolds with $\dim Y=n$, and $A$ be an ample line bundle on $Y$. If $m\geqslant1$ is an integer, then the following results hold.
\begin{enumerate}
\item[(1)] For any $q>0$, $i\geqslant0$ and $l>(m-1)(n+2)$,
\begin{equation*}
H^{q}(Y,R^{i}f_{\ast}(K^{m}_{X}\otimes\mathscr{I}(f,\|K^{m-1}_{X}\|))\otimes A^{l})=0.
\end{equation*}
\item[(2)] For any integer $j\geqslant0$ and a (non-zero) section $s$ of $f^{\ast}A^{j}$, the multiplication map induced by the tensor product with $s$
\begin{equation*}
\Phi:H^{q}(X,K^{m}_{X}\otimes f^{\ast}A^{l}\otimes\mathscr{I}(f,\|K^{m-1}_{X}\|))\rightarrow H^{q}(X,K^{m}_{X}\otimes f^{\ast}A^{l+j}\otimes\mathscr{I}(f,\|K^{m-1}_{X}\|))
\end{equation*}
is (well-defined and) injective for any $q\geqslant0$ and $l>(m-1)(n+2)$.
\end{enumerate}
\end{theorem}

This paper is organised as follows. We first recall some background materials in Section 2, including the asymptotic multiplier ideal sheaf, the definition of $GV$-sheaves in the sense of Pareschi and Popa \cite{PP11} and so on. Then, we prove Theorems \ref{t12} and \ref{t13} in Section \ref{sec:global}, Theorem \ref{t14}
in Section \ref{sec:generic}, and Theorem \ref{t15} in Section \ref{sec:further}.

\section{Preliminary}
\label{sec:preliminary}
In this section we introduce some basic materials. Assume that $f:X\rightarrow Y$ is a fibration between two projective manifolds, and $L$ is a holomorphic line bundle on $X$. Moreover, $L^{k}$ refers to the $k$-th tensor power with the convention that $L^{0}=\mathcal{O}_{X}$ and $L^{k}=(L^{\ast})^{-k}$ for $k<0$.

\subsection{The asymptotic multiplier ideal sheaf}
\label{sec:asymptotic}
This part is mostly collected from \cite{Laz04b}.

First recall the definition of the multiplier ideal sheaf associated to an ideal sheaf $\mathfrak{a}\subset\mathcal{O}_{X}$ and a positive real number $c$. Let $\mu:\tilde{X}\rightarrow X$ be a smooth modification such that $\mu^{\ast}\mathfrak{a}=\mathcal{O}_{\tilde{X}}(-E)$, where $E$ has the simple normal crossing support. Then the multiplier ideal sheaf is defined as
\begin{equation*}
\mathscr{I}(c\cdot\mathfrak{a}):=\mu_{\ast}\mathcal{O}_{\tilde{X}}(K_{\tilde{X}/X}-\lfloor cE\rfloor).
\end{equation*}
Here $\lfloor E\rfloor$ means the round-down.

Now suppose that $L$ is a line bundle on $X$ whose restriction to a general fibre of $f$ has non-negative Iitaka dimension. For a positive integer $k$, there is a naturally defined homomorphism
\begin{equation*}
\rho_{k}:f^{\ast}f_{\ast}L^{k}\rightarrow L^{k}.
\end{equation*}
The relative base-ideal $\mathfrak{a}_{k,f}$ of $|L^{k}|$ is then defined as the image of the induced homomorphism
\begin{equation*}
f^{\ast}f_{\ast}L^{k}\otimes L^{-k}\rightarrow\mathcal{O}_{X}.
\end{equation*}
Hence for a given positive real number $c$, we have the multiplier ideal sheaf $\mathscr{I}(\frac{c}{k}\cdot\mathfrak{a}_{k,f})$
which is also denoted by $\mathscr{I}(f,\frac{c}{k}|L^{k}|)$. Hence
\begin{equation*}
\mathscr{I}\bigl(f,\frac{c}{k}|L^{k}|\bigr)\subseteq\mathcal{O}_{X}.
\end{equation*}
It is not hard to verify that for every integer $p\geqslant1$ one has the inclusion
\begin{equation*}
\mathscr{I}\bigl(f,\frac{c}{k}|L^{k}|\bigr)\subseteq\mathscr{I}\bigl(f,\frac{c}{pk}|L^{pk}|\bigr).
\end{equation*}
Therefore the family of ideals
\begin{equation*}
\bigl\{\mathscr{I}\bigl(f,\frac{c}{k}|L^{k}|\bigr)\bigr\}_{(k\geqslant0)}
\end{equation*}
has a unique maximal element from the ascending chain condition on ideals.

\begin{definition}\label{d21}
The relative asymptotic multiplier ideal sheaf associated to $f$, $c$ and $|L|$,
\begin{equation*}
\mathscr{I}(f,c\|L\|)
\end{equation*}
is defined to be the unique maximal member among the family of ideals $\{\mathscr{I}(f,\frac{c}{k}|L^{k}|)\}_{(k\geq 0)}$.
\end{definition}

Next, we explain the analytic counterpart of the relative multiple ideal sheaf. By definition,
\begin{equation*}
\mathscr{I}\bigl(f,c\|L\|\bigr)=\mathscr{I}\bigl(f,\frac{c}{k}|L^{k}|\bigr)=\mathscr{I}\bigl(\frac{c}{k}\cdot\mathfrak{a}_{k,f}\bigr)
\end{equation*}
for some $k$. In this case, we will say that $k$ computes $\mathscr{I}(f,c\|L\|)$. Let $U$ be a local coordinate ball of $Y$. By definition, we can pick $\{u_{1},...,u_{m}\}$ in $\Gamma(f^{-1}(U),L^{k})$ which generate $\mathfrak{a}_{k,f}$ on $f^{-1}(U)$. Let $\varphi_{U}=\log(|u_{1}|^{2}+\cdots+|u_{m}|^{2})$ which is a singular metric on $L^{k}|_{f^{-1}(U)}$. We verify that
\begin{equation*}
\mathscr{I}(f,\frac{c}{k}|L^{k}|)=\mathscr{I}(\frac{c}{k}\varphi_U)\textrm{ on }f^{-1}(U).
\end{equation*}

Indeed, let $\mu:\tilde{X}\rightarrow X$ be a smooth modification of $\mathfrak{a}_{k,f}$. Then $\mu^{\ast}\mathfrak{a}_{k,f}=\mathcal{O}_{\tilde{X}}(-E)$ such that $E+\textrm{except}(\mu)$ has the simple normal crossing support. Here $\textrm{except}(\mu)$ is the exceptional divisor of $\mu$. Now it is computed in \cite{Dem12} that
\begin{equation*}
\mathscr{I}\bigl(\frac{c}{k}\varphi_U\bigr)=\mu_{\ast}\mathcal{O}_{\tilde{X}}\bigl(K_{\tilde{X}/X}-\lfloor\frac{c}{k}E\rfloor\bigr)\textrm{ on }f^{-1}(U)
\end{equation*}
which coincides with the definition of $\mathscr{I}(f,\frac{c}{k}|L^{k}|)$. Furthermore, if $v_{1},...,v_{m}$ are alternative generators and $\psi_{U}=\log(|v_{1}|^{2}+\cdots+|v_{m}|^{2})$, obviously we have $\mathscr{I}(\frac{c}{k}\varphi_{U})=\mathscr{I}(\frac{c}{k}\psi_{U})$. Hence all the $\mathscr{I}(\frac{c}{k}\varphi_{U})$ patch together to give a globally defined multiplier ideal sheaf $\mathscr{I}(\frac{c}{k}\varphi)$ such that
\begin{equation*}
\mathscr{I}(\frac{c}{k}\varphi)=\mathscr{I}(f,\frac{c}{k}|L^{k}|)=\mathscr{I}(f,c\|L\|).
\end{equation*}
Note that $\{f^{-1}(U),\frac{1}{k}\varphi_{U}\}$ does not give a globally defined metric on $L$ in general. The $\frac{1}{k}\varphi$ is interpreted as the collection of functions $\{f^{-1}(U),\frac{1}{k}\varphi_{U}\}$ by abusing the notation, which is called the collection of (local) singular metrics on $L$ associated to $\mathscr{I}(f,c\|L\|)$. Certainly it depends on the choice of $k$ and is not unique.

The following elementary property is collected from \cite{Laz04b}.
\begin{proposition}[(c.f. Proposition 11.2.15, \cite{Laz04b})]\label{p21}
Let $f:X\rightarrow Y$ be a fibration between two projective manifolds, and $L$ be a line bundle on $X$. Given a positive integer $k$, let
\begin{equation*}
\mathfrak{a}_{k,f}=\mathfrak{a}(f,|L^{k}|)
\end{equation*}
be the relative base-ideal of $|L^{k}|$ relative to $f$. Then the canonical map $\rho_{k}:f^{\ast}f_{\ast}L^{k}\rightarrow L^{k}$ factors through the inclusion $L^{k}\otimes\mathscr{I}(f,\|L^{k}\|)$, i.e.
\begin{equation*}
\mathfrak{a}_{k,f}\subseteq\mathscr{I}(f,\|L^{k}\|).
\end{equation*}
Equivalently, the natural map
\begin{equation*}
f_{\ast}(L^{k}\otimes\mathscr{I}(f,\|L^{k}\|))\rightarrow f_{\ast}(L^{k})
\end{equation*}
is an isomorphism.
\end{proposition}

\subsection{GV-sheaves in the sense of Pareschi and Popa}
\label{sec:gv}

\subsubsection{Definition}
We concentrate now on the case of more specific morphisms $f:X\rightarrow Y$, where $X$ is a projective manifold and $Y$ is an abelian variety. We denote by $P$ the normalised Poincar\'{e} bundle on the product $Y\times\textrm{Pic}^{0}(Y)$, and by $P_{\alpha}$ its restriction to the slice $Y\times\{\alpha\}$; this is of course just a different name for the point $\alpha\in\textrm{Pic}^{0}(Y)$.

\begin{definition}[(c.f. \cite{PP11})]\label{d22}
A coherent sheaf $\mathcal{F}$ on $X$ is said to be a $GV$-sheaf if
\begin{equation*}
\textrm{codim}_{\textrm{Pic}^{0}(Y)}\{\alpha\in\textrm{Pic}^{0}(Y)\mid H^{q}(X,\mathcal{F}\otimes f^{\ast}P_{\alpha})\neq0\}\geqslant q
\end{equation*}
for every $q\geqslant0$.
\end{definition}

\subsubsection{A brief review of the former results}
If $f$ is generically finite, then $K_{X}$ is a $GV$-sheaf by a special case of the generic vanishing theorem of Green and Lazarsfeld \cite{GrL87}. This result was generalised by Hacon \cite{Hac04}, to the effect that for an arbitrary $f$ the higher direct images $R^{i}f_{\ast}(K_{X})$ are $GV$-sheaves on $Y$ for all $i\geq 0$. On the other hand, there exist simple examples showing that even when $f$ is generically finite, the powers $K^{m}_{X}$ with $m\geqslant2$ are not necessarily
$GV$-sheaves; see \cite{PP11}, Example 5.6. Therefore it is quite surprising that Popa and Schnell \cite{PS14} showed that $f_{\ast}(K^{m}_{X})$ are $GV$-sheaves on $Y$ for all $m\geqslant1$. They then asked whether the higher direct images are also $GV$-sheaves. We positively answer this question in some sense in Theorem \ref{t14}, i.e., we prove that $R^{i}f_{\ast}(K^{m}_{X}\otimes\mathscr{I}(f,\|K^{m-1}_{X}\|))$ are $GV$-sheaves.

Note that $R^{i}f_{\ast}(K^{m}_{X})$ are not necessarily $GV$-sheaves. In fact, \cite{Shi16} constructed such a counterexample; see Example 4.5 there. After that T.~Shibata \cite{Shi16} (Proposition 4.8) proved that $R^{i}f_{\ast}(K^{m}_{X})$ are $GV$-sheaves with the assumptions that $\dim X=2$ and $\kappa(X)\geqslant0$. We will make a brief illustration (see {\it Remark} \ref{r41}) that our Theorem \ref{t14} actually implies this result. Therefore, it is quite natural to consider the twist by the asymptotic multiplier ideal.

\subsection{Nef coherent sheaves}
\label{sec:coherent}
To any coherent sheaf $\mathcal{F}$ on a projective manifold $Y$, one associates the scheme \cite{Har77}
\[
\mathbb{P}(\mathcal{F}):=\textrm{Proj}(\oplus_{m\geqslant0}\textrm{Sym}^{m}\mathcal{F})
\]
and an inevitable sheaf $\mathcal{O}_{\mathbb{P}(\mathcal{F})}(1)$ on $\mathbb{P}(\mathcal{F})$. Then we have the following definition.

\begin{definition}\label{d23}
A coherent sheaf $\mathcal{F}$ is said to be nef if $\mathcal{O}_{\mathbb{P}(\mathcal{F})}(1)$ is.
\end{definition}

\section{Main theorem}
\label{sec:global}

\subsection{The vanishing theorem}
\label{sec:vanishing}
We first prove Theorem \ref{t13}. We will apply  the following Koll\'{a}r-type vanishing theorem.
\begin{theorem}[(c.f. \cite{FuM16}, Theorem D)]\label{t31}
Let $f:X\rightarrow Y$ be a surjective morphism from a compact K\"{a}hler manifold $X$ onto a projective variety $Y$.
Let $F$ be a holomorphic line bundle on $X$ with a (singular) metric $\varphi_{F}$ such that $i\Theta_{F,\varphi_{F}}\geqslant0$. Let $N$ be a holomorphic line bundle on $X$. Assume that there exist two positive integers $a$ and $b$ and an ample line bundle $A$ on $Y$ such that $N^{a}=f^{\ast}A^{b}$. Then
\begin{equation*}
H^{q}(Y,R^{i}f_{\ast}(K_{X}\otimes F\otimes\mathscr{I}(\varphi_{F})\otimes N))=0
\end{equation*}
for every $q>0$ and $i\geqslant0$.
\end{theorem}

We can not pick $K_X^{m-1}$ as $F$ in Theorem \ref{t31}  since in general we can not obtain a metric $\psi$ on $K_{X}$ such that \begin{equation*}
i\Theta_{K_{X},\psi}\geqslant0\quad\textup{and}\quad \mathscr{I}((m-1)\psi)=\mathscr{I}(f,\|K^{m-1}_{X}\|).
\end{equation*}
Instead we consider $F=K_{X}^{m-1}\otimes f^{\ast}A^{(m-1)(n+1)}$, where $n=\dim Y$.

\begin{proof}[Proof of Theorem \ref{t13}]
Let $\{U_{\alpha}\}$ be a finite local coordinate chart of $Y$. Let $p$ be a divisible and large enough integer which computes $\mathscr{I}(f,\|K^{m-1}_{X}\|)$, and let $\varphi=\{\varphi_{\alpha}\}$ be the associated metrics. So $\varphi_{\alpha}$ is a singular metric on $K_{X}|_{f^{-1}(U_{\alpha})}$ of the form that
\begin{equation*}
\varphi_{\alpha}=\frac{1}{p}\log\sum_{i}|u_{i,\alpha}|^{2},
\end{equation*}
where $u_{i,\alpha}\in\Gamma(f^{-1}(U_{\alpha}),K^{p}_{X})$. In particular, $\{u_{i,\alpha}\}$ are local generators of the relative base-ideal $\mathfrak{a}_{p,f}$ of $|K^{p}_{X}|$ and
\begin{equation*}
\mathscr{I}((m-1)\varphi)=\mathscr{I}(f,\|K^{m-1}_{X}\|).
\end{equation*}

Now $f_{\ast}(K^{p}_{X})\otimes A^{p(n+1)}$ is globally generated by Theorem \ref{t11}. Then there exist sections
\begin{equation*}
\{v_{ij,\alpha}\}\subseteq H^{0}(X,K^{p}_{X}\otimes f^{\ast}A^{p(n+1)})
\end{equation*}
such that $\log\sum|u_{i,\alpha}|^{2}$ and $\log\sum|v_{ij,\alpha}|^{2}$ are equivalent with respect to the singularities \cite{Dem12} on $f^{-1}(U_{\alpha})$.

In fact, if we denote, by abusing the notation,
\begin{equation*}
H^{0}(Y,f_{\ast}(K^{p}_{X})\otimes A^{p(n+1)})\quad\textrm{ and }\quad H^{0}(X,K^{p}_{X}\otimes f^{\ast}A^{p(n+1)})
\end{equation*}
the trivial vector bundles on $Y$ and $X$ respectively, the morphism
\begin{equation*}
H^{0}(Y,f_{\ast}(K^{p}_{X})\otimes A^{p(n+1)})\rightarrow f_{\ast}(K^{p}_{X})\otimes A^{p(n+1)}
\end{equation*}
as well as
\begin{equation*}
H^{0}(X,K^{p}_{X}\otimes f^{\ast}A^{p(n+1)})\rightarrow f^{\ast}f_{\ast}(K^{p}_{X})\otimes f^{\ast}A^{p(n+1)}
\end{equation*}
is surjective by definition. Therefore
\begin{equation*}
H^{0}(X,K^{p}_{X}\otimes f^{\ast}A^{p(n+1)})\otimes K^{-p}_{X}\rightarrow f^{\ast}f_{\ast}(K^{p}_{X})\otimes K^{-p}_{X}\otimes f^{\ast}A^{p(n+1)}
\end{equation*}
is also surjective. In particular, we have the following surjection:
\begin{equation}\label{e31}
H^{0}(X,K^{p}_{X}\otimes f^{\ast}A^{p(n+1)})\otimes K^{-p}_{X}\twoheadrightarrow\mathfrak{a}_{p,f}\otimes f^{\ast}A^{p(n+1)},
\end{equation}
where $\mathfrak{a}_{p,f}\otimes f^{\ast}A^{p(n+1)}$ by definition is the image of the natural morphism:
\begin{equation*}
f^{\ast}f_{\ast}(K^{p}_{X})\otimes K^{-p}_{X}\otimes f^{\ast}A^{p(n+1)}\rightarrow f^{\ast}A^{p(n+1)}.
\end{equation*}

Let $\{s_{j}\}$ be the set of global sections that generates $f^{\ast}A^{p(n+1)}$. Due to (\ref{e31}), all of the sections $\{u_{i,\alpha}\otimes s_{j}\}$ extend over $X$ as the global sections $\{v_{ij,\alpha}\}$ of
\begin{equation*}
H^{0}(X,K^{p}_{X}\otimes f^{\ast}A^{p(n+1)}).
\end{equation*}
Since $\{s_{j}\}$ generates $f^{\ast}A^{p(n+1)}$, $\log\sum|u_{i,\alpha}|^{2}$ and $\log\sum|v_{ij,\alpha}|^{2}$ are equivalent with respect to the singularities on $f^{-1}(U_{\alpha})$.

Now the sections $\{v_{ij,\alpha}\}$, as $i,j,\alpha$ vary, together define a (singular) metric $\chi$ on
\begin{equation*}
K^{p}_{X}\otimes f^{\ast}A^{p(n+1)}
\end{equation*}
with positive curvature current. Next we show that
\begin{equation*}
\mathscr{I}(p^{-1}(m-1)\chi)=\mathscr{I}(f,\|K^{m-1}_{X}\|).
\end{equation*}

Let $\mathfrak{a}$ be the ideal sheaf defined by $\{v_{ij,\alpha}\}$ (as $i,j,\alpha$ vary) and let $\mathfrak{a}_{\alpha}$ be the ideal sheaf (on $f^{-1}(U_{\alpha})$) defined by $\{u_{i,\alpha}\}$ (as $i$ varies). Then by the choice of $\{u_{i,\alpha}\}$, for every $\alpha$ we have
\begin{equation}\label{e32}
\mathscr{I}(p^{-1}(m-1)\cdot\mathfrak{a})=\mathscr{I}(p^{-1}(m-1)\cdot\mathfrak{a}_{\alpha})\textrm{ on }f^{-1}(U_{\alpha}).
\end{equation}
In fact, notice that $\mathfrak{a}_{\alpha}$ is just the relative base-ideal of $|K^{p}_{X}|$ restricted on $f^{-1}(U_{\alpha})$. Accordingly, $\mathscr{I}(p^{-1}(m-1)\cdot\mathfrak{a}_{\alpha})$ is the restriction of $\mathscr{I}(f,p^{-1}(m-1)|K^{p}_{X}|)$ on $f^{-1}(U_{\alpha})$. However,
\begin{equation*}
\mathscr{I}(f,p^{-1}(m-1)|K^{p}_{X}|)=\mathscr{I}(f,\|K^{m-1}_{X}\|)
\end{equation*}
is now the unique maximal element in
\begin{equation*}
\bigl\{\mathscr{I}\bigl(f,\frac{1}{k}|K^{k(m-1)}_{X}|\bigr)\bigr\}_{(k\geqslant0)}.
\end{equation*}
So $\mathscr{I}(f,p^{-1}(m-1)|K^{p}_{X}|)$, as well as $\mathscr{I}(p^{-1}(m-1)\cdot\mathfrak{a}_{\alpha})$ should be stable. In other words, if $u$ is a section of $\Gamma(f^{-1}(U_{\alpha}),K^{p}_{X})$, we must have
\begin{equation*}
\mathscr{I}(p^{-1}(m-1)\cdot\mathfrak{a}_{\alpha})=\mathscr{I}(f,p^{-1}(m-1)|K^{p}_{X}|)|_{f^{-1}(U_{\alpha})}
=\mathscr{I}(p^{-1}(m-1)\cdot(\mathfrak{a}_{\alpha}\cup(u))).
\end{equation*}
Here $\mathfrak{a}_{\alpha}\cup(u)$ refers to the ideal sheaf generated by both of $\mathfrak{a}_{\alpha}$ and $u$.

On the other hand, by construction the sections $\{v_{ij,\alpha}\}$, as $i,j$ vary, also generate $\mathfrak{a}_{\alpha}$ on $f^{-1}(U_{\alpha})$. By stability, the sections $\{v_{ij,\alpha}\}$, as $i,j,\alpha$ vary, will lead to the same multiplier ideal sheaf here. It implies (\ref{e32}). Equivalently,
\begin{equation*}
\mathscr{I}(p^{-1}(m-1)\chi)=\mathscr{I}(p^{-1}(m-1)\cdot\mathfrak{a})=\mathscr{I}(f,\|K^{m-1}_{X}\|)
\end{equation*}
on $f^{-1}(U_{\alpha})$ hence everywhere.

Now let
\begin{equation*}
 (F,\varphi_{F})=(K^{m-1}_{X}\otimes f^{\ast}A^{(m-1)(n+1)},p^{-1}(m-1)\chi).
\end{equation*}
Then as is shown before,
\begin{equation*}
\mathscr{I}(\varphi_{F})=\mathscr{I}(f,\|K^{m-1}_{X}\|).
\end{equation*}
Furthermore, since $l\geqslant (m-1)n+m$ by hypothesis, $l-(m-1)(n+1)\geqslant1$. Now let
\begin{equation*}
N=f^{\ast}A^{l-(m-1)(n+1)},
\end{equation*}
we have
\begin{equation*}
R^{i}f_{\ast}(K_{X}\otimes F\otimes\mathscr{I}(\varphi_{F})\otimes N)=R^{i}f_{\ast}(K^{m}_{X}\otimes\mathscr{I}(f,\|K^{m-1}_{X}\|))\otimes A^{l}.
\end{equation*}
Applying Theorem \ref{t31} (with the same notation there), we then obtain the desired vanishing result.
\end{proof}

\subsection{Global generation}
\label{sec:generation}
Using Theorem \ref{t13}, we can prove the global generation of higher direct images, namely Theorem \ref{t12}. We first review the definition and a basic result of the Castelnuovo--Mumford regularity \cite{Mum66}.

\begin{definition}\label{d31}
Let $X$ be a projective manifold and $L$ an ample and globally generated line bundle on $X$. Given an integer $m$, a coherent sheaf $F$ on $X$ is $m$-regular with respect to $L$ if for all $i\geqslant1$
\begin{equation*}
H^{i}(X,F\otimes L^{m-i})=0.
\end{equation*}
\end{definition}

\begin{theorem}(c.f. \cite{Mum66})\label{t32}
Let $X$ be a projective manifold and $L$ an ample and globally generated line bundle on $X$. If $F$ is a coherent sheaf on $X$ that is $m$-regular with respect to $L$, then the sheaf $F\otimes L^{m}$ is globally generated.
\end{theorem}

After this, we can prove Theorem \ref{t12}.
\begin{proof}[Proof of Theorem \ref{t12}]
Since $l\geqslant m(n+1)$, $l-q\geqslant (m-1)n+m$ when $q\leqslant n$. It then follows from Theorem \ref{t13} that for every $q\geqslant1$,
\begin{equation*}
  H^{q}(Y,R^{i}f_{\ast}(K^{m}_{X}\otimes\mathscr{I}(f,\|K^{m-1}_{X}\|))\otimes A^{l-q})=0.
\end{equation*}
Hence the sheaf $R^{i}f_{\ast}(K^{m}_{X}\otimes\mathscr{I}(f,\|K^{m-1}_{X}\|))\otimes A^{l}$ is $0$-regular with respect to $A$. So it is globally generated by Theorem \ref{t32}.
\end{proof}

\section{Generic vanishing}
\label{sec:generic}
We first prove that the higher direct images are $GV$-sheaves.
\begin{proof}[Proof of Theorem \ref{t14}]
In view of \cite{PP11,PS14}, it is enough to show that for every finite \'{e}tale morphism $\beta:Z\rightarrow Y$ of abelian varieties and an ample and globally generated line bundle $H$ on $Z$, we have
\begin{equation}\label{001}
H^{q}(Z,H^{l}\otimes \beta^{\ast}R^{i}f_{\ast}(K^{m}_{X}\otimes\mathscr{I}(f,\|K^{m-1}_{X}\|)))=0
\end{equation}
for every $q>0$ with $l$ large enough. Here {\sl large enough} means that  there exists a bound $d$ depending only on $n$ and $m$ such that vanishing (\ref{001}) holds for any $Z$ and $H$ as long as $l\geqslant d$. In particular, we cannot apply the Serre asymptotic vanishing theorem \cite{Har77} here.

Now we prove this vanishing result. Put $W:=Z\times_{Y}X$ the fibre product \cite{Har77}. Then we have the following commutative diagram:
\begin{equation*}
\begin{CD}
W@>\alpha>>X&\\
@V{g}VV @VV{f}V&\\
Z@>>\beta>Y&\ .
\end{CD}
\end{equation*}
By construction, $\alpha$ is also finite and \'{e}tale. Hence we have $\alpha^{\ast}K_{X}=K_{W}$ and
\begin{equation*}
\alpha^{\ast}\mathscr{I}(f,\|K^{m-1}_{X}\|)=\mathscr{I}(g,\|K^{m-1}_{W}\|)
\end{equation*}
by the behaviour of asymptotic multiplier ideals under \'{e}tale covers (c.f. \cite{Laz04b}, Theorem 11.2.16).

By the flat base change theorem \cite{Har77},
\begin{equation*}
\beta^{\ast}R^{i}f_{\ast}(K^{m}_{X}\otimes\mathscr{I}(f,\|K^{m-1}_{X}\|))\simeq R^{i}g_{\ast}(K^{m}_{W}\otimes\mathscr{I}(f,\|K^{m-1}_{W}\|)).
\end{equation*}
Hence we obtain vanishing (\ref{001}) by Theorem \ref{t13}.
\end{proof}

\begin{rem}\label{r41}
A variant of Theorem \ref{t14} can actually recover the following proposition, which is originally obtained in \cite{Shi16}.
\end{rem}

\begin{proposition}[(c.f. \cite{Shi16}, Proposition 4.8)]
Let $f:X\rightarrow Y$ be a morphism from a smooth projective surface $X$ to an abelian variety $Y$. Assume that $\kappa(X)\geqslant0$. Then $R^{i}f_{\ast}(K^{m}_{X})$ are $GV$-sheaves for every $i\geqslant0$ and $m\geqslant1$.
\begin{proof}[Reproof via Theorem \ref{t13}]
The proof is similar to the one in \cite{Shi16}, hence we only sketch it.

We may assume without loss of generality that $\dim f(X)\geqslant1$. Then it is enough to show that $R^{1}f_{\ast}(K^{m}_{X})$ is a $GV$-sheaf. Since $\kappa(X)\geqslant0$, we can take a series of contractions of $(-1)$-curves $\varepsilon:X\rightarrow X^{\prime}$ such that $K_{X^{\prime}}$ is semi-ample, and obtain a natural morphism $f^{\prime}:X^{\prime}\rightarrow Y$ such that $f^{\prime}\circ\varepsilon=f$. Combined with the Leray spectral sequence \cite{Har77}, it is left to prove that
\begin{equation}\label{e42}
H^{q}(Y,f^{\prime}_{\ast}R^{1}\varepsilon_{\ast}(K^{m}_{X})\otimes H^{l})=0
\end{equation}
and
\begin{equation}\label{e43}
H^{q}(Y,R^{1}f^{\prime}_{\ast}(K^{m}_{X^{\prime}})\otimes H^{l})=0
\end{equation}
for every $q>0$ and any ample and globally generated line bundle $H$ on $Y$ with $l$ {\sl large enough}. The vanishing (\ref{e42}) is quite obvious for the reason of dimension, while (\ref{e43}) is a direct consequence of Theorem \ref{t13} since  the semi-ampleness implies that $\mathscr{I}(f^{\prime},\|K^{m-1}_{X^{\prime}}\|)=\mathcal{O}_{X^{\prime}}$.
\end{proof}
\end{proposition}

Once we know generic vanishing the situation is in fact much better than what we obtained for morphisms to arbitrary varieties.

\begin{proof}[Proof of Corollary \ref{c11}]
In view of \cite{PS14}, Corollary 5.4, everything is immediately verified once we know that $R^{i}f_{\ast}(K^{m}_{X}\otimes\mathscr{I}(f,\|K^{m-1}_{X}\|))$ is a $GV$-sheaf. Hence we omit the proof here.
\end{proof}

\section{Further discussions}
\label{sec:further}
In this section, we prove Theorem \ref{t15} based on results in \cite{Den21,DuM19,Dut20,Iwa20}. More precisely, we will apply the following theorem.

\begin{theorem}[(c.f. \cite{Iwa20}, Theorem 1.4)]\label{t61}
Let $f:X\rightarrow Y$ be a fibration between two projective manifolds with $\dim Y=n$, and $A$ be an ample line bundle on $Y$.
For any integer $m\geqslant1$ and $l\geqslant\frac 1 2 n(n-1)+m(n+1)$, the sheaf
\begin{equation*}
f_{\ast}(K^{m}_{X})\otimes A^{l}
\end{equation*}
is generated by the global sections at a regular value $y$ of $f$.
\end{theorem}

The proof of Theorem \ref{t15} then involves the same argument as Theorem \ref{t13} with slightly adjustment.
\begin{proof}[Proof of Theorem \ref{t15}]
(1) Let $\{U_{\alpha}\}$ be a finite local coordinate chart of $Y$. Let $p$ be a divisible and large enough integer which computes $\mathscr{I}(f,\|K^{m-1}_{X}\|)$. Moreover, $f^{\ast}A^{p(n+2)}$ is globally generated and
\begin{equation}\label{002}
p\geqslant\frac 1 2 n(n-1).
\end{equation}
Let $\varphi=\{\varphi_{\alpha}\}$ be the metrics associated with $\mathscr{I}(f,\|K^{m-1}_{X}\|)$. Then $\varphi_{\alpha}$ is a singular metric on $K_{X}|_{f^{-1}(U_{\alpha})}$ of the form that
\begin{equation*}
\varphi_{\alpha}=\frac{1}{p}\log\sum|u_{i,\alpha}|^{2},
\end{equation*}
where $u_{i,\alpha}\in\Gamma(f^{-1}(U_{\alpha}),K^{p}_{X})$. Moreover,
\begin{equation*}
\mathscr{I}((m-1)\varphi)=\mathscr{I}(f,\|K^{m-1}_{X}\|).
\end{equation*}

We write inequality  (\ref{002}) as
\begin{equation*}
p(n+2)\geqslant \frac 1 2 n(n-1)+p(n+1).
\end{equation*}
So by Theorem \ref{t61} the sheaf
\begin{equation*}
f_{\ast}(K^{p}_{X})\otimes A^{p(n+2)}
\end{equation*}
is globally generated on $Y$. Now let $\{s_{j}\}$ be the set of global sections that generates $f^{\ast}A^{p(n+2)}$. Then as is shown in the proof of Theorem \ref{t13}, all of the sections $\{u_{i,\alpha}\otimes s_{j}\}$ extend over $X$ as the sections $\{v_{ij,\alpha}\}$ in
\begin{equation*}
H^{0}(X,K^{p}_{X}\otimes f^{\ast}A^{p(n+2)}).
\end{equation*}
The sections $\{v_{ij,\alpha}\}$, as $i,j,\alpha$ vary, together define a (singular) metric $\chi$ on
\begin{equation*}
K^{p}_{X}\otimes f^{\ast}A^{p(n+2)}
\end{equation*}
with positive curvature current. Then by the same argument as Theorem \ref{t13}, we obtain
\begin{equation*}
\mathscr{I}(p^{-1}(m-1)\chi)=\mathscr{I}(f,\|K^{m-1}_{X}\|).
\end{equation*}

Let
\begin{equation*}
 (F,\varphi_{F})=(K^{m-1}_{X}\otimes f^{\ast}A^{(m-1)(n+2)},p^{-1}(m-1)\chi).
\end{equation*}
As is shown before,
\begin{equation*}
\mathscr{I}(\varphi_{F})=\mathscr{I}(f,\|K^{m-1}_{X}\|).
\end{equation*}
Furthermore, since $l>(m-1)(n+2)$ by hypothesis, $l-(m-1)(n+2)\geqslant1$. So let
\begin{equation*}
N=f^{\ast}A^{l-(m-1)(n+2)}.
\end{equation*}
We have
\begin{equation*}
R^{i}f_{\ast}(K_{X}\otimes F\otimes\mathscr{I}(\varphi_{F})\otimes N)=R^{i}f_{\ast}(K^{m}_{X}\otimes\mathscr{I}(f,\|K^{m-1}_{X}\|))\otimes A^{l}.
\end{equation*}
Then by Theorem \ref{t31} (with the same notation there), we obtain the desired vanishing result.

(2) The strategy is to apply the following injectivity theorem in \cite{Mat14}.

\begin{theorem}[(c.f. \cite{Mat14}, Theorem 1.5)]\label{t52}
Let $(F,\varphi_{F})$ and $(M,\varphi_{M})$ be line bundles with (singular) metrics on a compact K\"{a}hler manifold $X$. Assume the following conditions:
\begin{enumerate}
\item[(a)] There exists a subvariety $Z$ on $X$ such that $\varphi_{F}$ and $\varphi_{M}$ are smooth on $X\setminus Z$;
\item[(b)] $i\Theta_{F,\varphi_{F}}\geqslant\gamma$ and $i\Theta_{M,\varphi_{M}}\geqslant\gamma$ on $X$ for some smooth $(1,1)$-form $\gamma$ on $X$;
\item[(c)] $i\Theta_{F,\varphi_{F}}\geqslant\varepsilon i\Theta_{M,\varphi_{M}}$ for some positive number $\varepsilon>0$.
\end{enumerate}
Then for a (non-zero) section $s$ of $M$ with $\sup_{X}|s|_{\varphi_{M}}<\infty$, the multiplication map induced by the tensor product with $s$
\begin{equation*}
\Phi_{s}:H^{q}(X,K_{X}\otimes F\otimes\mathscr{I}(\varphi_{F}))\rightarrow H^{q}(X,K_{X}\otimes F\otimes M\otimes\mathscr{I}(\varphi_{F}+\varphi_{M}))
\end{equation*}
is (well-defined and) injective for any $q$.
\end{theorem}

From (1) we know that
\begin{equation*}
(K^{m-1}_{X}\otimes f^{\ast}A^{(m-1)(n+2)},p^{-1}(m-1)\chi)
\end{equation*}
is a Hermitian line bundle with positive curvature current and satisfies
\begin{equation*}
\mathscr{I}(p^{-1}(m-1)\chi)=\mathscr{I}(f,\|K^{m-1}_{X}\|).
\end{equation*}
In particular, $\chi$ is smooth outside a subvariety. Let $\psi$ be a smooth metric on $f^{\ast}A$ with positive curvature. Let
\begin{equation*}
 (F,\varphi_{F})=(K^{m-1}_{X}\otimes f^{\ast}A^{l},p^{-1}(m-1)\chi+(l-(m-1)(n+2))\psi)
 \end{equation*}
and
\begin{equation*}
(M,\varphi_{M})=(f^{\ast}A^{j},j\psi).
\end{equation*}
Since $l>(m-1)(n+2)$ by hypothesis, $l-(m-1)(n+2)\geqslant1$. Therefore let $\varepsilon=\frac{1}{j}$, we have
\begin{equation*}
i\Theta_{F,\varphi_{F}}\geqslant\varepsilon i\Theta_{M,\varphi_{M}}\geqslant0.
\end{equation*}
Moreover, since $\psi$ is smooth,
\begin{equation*}
\mathscr{I}(\varphi_{F}+\varphi_{M})=\mathscr{I}(\varphi_{F})=\mathscr{I}(f,\|K^{m-1}_{X}\|).
\end{equation*}
So the proof is finished by directly applying Theorem \ref{t52}.
\end{proof}

\begin{acknowledgements}
The authors would like to thank Prof. Mihnea Popa for pointing out the counterexample in \cite{Shi16} which shows that the higher direct images to abelian varieties are not necessarily $GV$.
\end{acknowledgements}

\end{document}